\documentclass{amsart}
\usepackage{tikz,tikz-cd}
\usetikzlibrary{matrix, arrows}
\usepackage{amssymb,amsmath,amsthm}
\usepackage{graphicx}
\usepackage{bbold,accents}
\title[Algebroids and Twistor Spaces]{Algebroids and Twistor Spaces}
\author{Steven Gindi}
\newtheorem{thm}{Theorem}[section]
\newtheorem{lemma}[thm]{Lemma}
\newtheorem{prop}[thm]{Proposition}

\newtheorem{rmk}[thm]{Remark} 
 
\newtheorem{nota}[thm]{Notation}
\newtheorem{note}[thm]{Note}
\theoremstyle{definition} \newtheorem{example}[thm]{Example}
\theoremstyle{definition}  \newtheorem{defi}[thm]{Definition}
\numberwithin{equation}{section}
\begin{document}
\begin{large}
\maketitle
\vspace*{-.45cm}
\tableofcontents

\section{Introduction}

In my previous papers, I introduced a new way of studying bihermitian manifolds via their twistor spaces (see \cite{Gindi1,Gindi2} and \cite{Rocek1, Hitchin1, Gualt1, Apost1}). The main idea was to first demonstrate that the twistor space of a bihermitian manifold admits integrable complex structures and holomorphic sections. I then used these sections to geometrically interpret certain holomorphic objects in the bihermitian manifold as intersections of holomorphic subvarieties inside the twistor space. This point of view led, for instance, to new insights into the holomorphic Poisson structures on bihermitian manifolds. 


The purpose of this paper and \cite{Gindi5} is to develop and enhance my twistor approach to bihermitian geometry. To do so, I will first present in this paper various skew algebroids on twistor spaces of vector spaces and will describe the corresponding foliations. These structures will be used in \cite{Gindi5} to build skew algebroids on twistor bundles fibered over bihermitian manifolds. By studying certain embeddings, I will then determine, for instance, how a bihermitian manifold interacts with the different associated foliations of its twistor space. This will consequently lead to results about the local geometry of bihermitian manifolds. Other applications, including those to Lie groups, will be given in \cite{Gindi5} and \cite{Gindi6}. (Note that skew algebroids are related to Lie algebroids and Poisson structures; see Definition \ref{defSA1} and \cite{Dufour}.)

 A general twistor space that I consider in this paper is the space of complex structures of a vector space that are compatible with some metric. While my algebroid constructions are based, in particular, on using a certain diagonalizable bundle map over the twistor space. They are also closely related to a representation theory of a pair of  complex structures that I have worked on in \cite{Gindi3}.  

 In Section \ref{secFTSA}, I give my first set of skew algebroids. The leaves of the corresponding foliations of the twistor space are isomorphic to Schubert cells in a certain homogeneous maximal isotropic space. Using a  diagonalizable bundle map, to be specified below, I then construct in Section \ref{secMA} a plethora of other algebroids on the twistor space. All of these algebroids yield rich foliations of the Schubert cells. 

Letting $(V,g)$ be a vector space with a metric, the first step in building these algebroids is to fix a complex structure $J$ in the associated twistor space, $\mathcal{T}$. The next step is to consider the diagonalizable bundle map $\{J,\phi\} \in \Gamma(\mathfrak{gl}(V) \times \mathcal{T})$, where $\phi|_{K}=K$. I have then constructed  algebroids whose integrable distributions interact, as endomorphisms, with the eigenspaces of $\{J,\phi \}$ in different ways (see Sections \ref{secALRES}, \ref{secSPE}, \ref{secID2} and \ref{secICD}). To do so, I have used, in particular, polynomials in the map $\{J, \phi \}$ of the form $Q= a_{k}\{J,\phi\}^{k}+ ...+a_{1}\{J,\phi\}+ a_{0}\mathbb{1}$, where the $a_{i}$ are \textit{functions} on $\mathcal{T}$ that satisfy certain differential equations. As an example, the $a_{i}$ can be chosen to be any elements in $\mathbb{R}[f_{1},...,f_{2n}]$, where $f_{j}|_{K}$ is the $j^{th}$ symmetric polynomial in the eigenvalues of $\{J,K\}$.

These algebroids then lead to an abundance of generalized foliations of $\mathcal{T}$, which I describe in detail in Sections \ref{secFCF} and \ref{secSCF}. The leaves of all the foliations are smooth embedded submanifolds and the leaves of a large class of them are complex submanifolds. Moreover, the manifold type of the leaves passing through $K\in \mathcal{T}$ depends on the subset of eigenvalues of $\{J, \phi\}|_{K}$ that are roots of the  chosen polynomial $Q|_{K}$ (Section \ref{secCMTQS}, Example \ref{exCMT}). As for the foliations that I introduce in Section \ref{secSPEFOL}, the manifold type of the leaves depends on whether the eigenvalues of $\{J, \phi\}$ satisfy certain \textit{systems of polynomial equations} (see especially Example \ref{exSPE}).  

In Section \ref{secFW}, I then describe some additional results that I will present in an updated version of this paper \cite{Gindi4}. 

\begin{rmk}
\label{rmkUPDATED}
This paper serves as an announcement of my results and contains only a few proofs. I will present all the proofs in \cite{Gindi4}.
\end{rmk} 

\section{Preliminaries}
\subsection{Twistor Spaces} 
\label{secBTS}
Let $(V,g)$ be a $2n$ dimensional real vector space equipped with a positive definite metric and let $\mathcal{T}(V,g)=\{J \in EndV | \ J^{2}=-1, \ g(J\cdot,J\cdot)=g(\cdot,\cdot)\}$ be one of its twistor spaces. To describe some of the properties of $\mathcal{T}(V,g)$, consider the action of $O(V,g)$ on $EndV$ via conjugation: $B \cdot A= BAB^{-1}.$  As  $\mathcal{T}(V,g)$ is a particular orbit of this action, it is isomorphic to 
\[O(V,g)/U(I),\] where $I\in \mathcal{T}(V,g)$ and $U(I)= \{B \in O(V,g)| \ [B,I]=0\} \cong U(n)$. It then follows that the dimension of $\mathcal{T}(V,g)$ is $n(n-1)$ and that if we consider $\mathcal{T}$ as a submanifold of $EndV$ then \[ T_{J}\mathcal{T}= [\mathfrak{o}(V,g),J]=\{A \in \mathfrak{o}(V,g) | \ \{A,J\}=0 \}.\]
With this, we may define a natural almost complex structure on $\mathcal{T}(V,g)$ that is well known to be integrable:
\[I_{\mathcal{T}}A = JA,  \text{ for } A \in T_{J}\mathcal{T}.\] 
\begin{nota}
 As was done above and will be continued below, we will at times denote $\mathcal{T}(V,g)$ by $\mathcal{T}(V)$ or just by $\mathcal{T}$.
\end{nota}

The goal then is to construct skew algebroids on $\mathcal{T}$. Here is the definition.

\begin{defi}
\label{defSA1}
\mbox{}
A \textit{skew algebroid} $(E,[\cdot,\cdot],\psi)$ on $\mathcal{T}(V,g)$ is a vector bundle equipped with a skew bracket on $\Gamma(E)$ together with a bundle map $\psi: E \longrightarrow T\mathcal{T}$ that satisfies
\begin{align*}
&a)\ \psi([A,B])= [\psi(A),\psi(B)]_{Lie}\\
&b) \ [A,fB]= f[A,B] + \psi(A)[f]B, 
\end{align*}
for all $A,B\in \Gamma(E)$ and $f\in C^{\infty}(\mathcal{T})$. 
\end{defi}

One of the main consequences of having skew algebroids is that they yield integrable distributions on $\mathcal{T}$.

\begin{prop}
Given a skew algebroid $(E,[\cdot,\cdot], \psi)$ on $\mathcal{T}$, the $Im\psi$ is an integrable distribution. 
\end{prop}

\begin{rmk}
The leaves of the associated foliation of $\mathcal{T}$ will in general have varied dimensions.
\end{rmk}

We will now specify the bundles and connections that we will be using to construct our algebroids. First consider the following vector bundles fibered over $\mathcal{T}$:
\[ \mathbb{V}:= V \times \mathcal{T} \text{ \ \ and \ \ } \mathfrak{o}(\mathbb{V},g):= \mathfrak{o}(V,g) \times \mathcal{T}.\]
To specify the others, let $\phi$ be the section of $\mathfrak{o}(\mathbb{V},g)$ defined by $\phi|_{K}=K$. 

\begin{defi}
\mbox{}

\vspace{.1cm}
 For $J \in \mathcal{T}(V)$, define
\begin{itemize}
\item[1)] $\mathfrak{gl}(\mathbb{V})= End(\mathbb{V})$
\item[2)] $\mathfrak{u}(\mathbb{V},J)= \{A \in \mathfrak{o}(\mathbb{V},g)\ | \ [A,J]=0 \}$
\item[3)] $\mathfrak{u}(\mathbb{V},\phi)= \{A \in \mathfrak{o}(\mathbb{V},g)\ | \ [A,\phi]=0 \}$
\item[4)]$\mathfrak{o}_{\{J\}}(\mathbb{V})= \{A \in \mathfrak{o}(\mathbb{V},g) \ |\  \{A,J\}=0 \}$
\item[5)] $\mathfrak{o}_{\{\phi\}}(\mathbb{V})= \{A \in \mathfrak{o}(\mathbb{V},g) \ | \ \{A,\phi\}=0 \}.$
\end{itemize}
In the above, one should evaluate both sides of the equations at a $K\in \mathcal{T}$. 
\end{defi}

\begin{nota}
 At times, we will denote $\mathfrak{u}(\mathbb{V},\phi)$ by $\mathfrak{u}(\phi)$ and $\mathfrak{o}_{\{\phi\}}(\mathbb{V})$ by $\mathfrak{o}_{\{\phi\}}$. Moreover, we will respectively denote any particular fiber of $\mathfrak{u}(\mathbb{V},J)$ and $\mathfrak{o}_{\{J\}}(\mathbb{V})$ by $\mathfrak{u}(V,J)$ and $\mathfrak{o}_{\{J\}}(V)$.  
\end{nota}

\begin{rmk}
Using this notation, $T\mathcal{T}$ will be identified with $[\mathfrak{o}(\mathbb{V},g),\phi]=\mathfrak{o}_{\{\phi\}}$.
\end{rmk}
Given these bundles, we will also be using the connection $\nabla=d+\frac{1}{2}(d\phi)\phi$ on $\mathbb{V}$, where $d$ is the trivial one. 

\begin{prop}
$\nabla$ is a torsionless connection on $T\mathcal{T}=[\mathfrak{o}(\mathbb{V},g),\phi]$ and satisfies $\nabla \phi=0$ and $\nabla I_{\mathcal{T}}=0$.
\end{prop}

\subsection{Background on a Pair of Complex Structures}
\label{secBOPCS}
Before we construct our algebroids, we will introduce certain decompositions of $V$ that are induced by two complex structures $J,K \in \mathcal{T}(V,g)$. These  decompositions encode the different possible algebraic interactions between a pair of complex structures and are closely related to a representation theory of an algebra which we call $iD_{\infty}$. In Section \ref{secMOT} we will describe how these decompositions have led us to build various algebroids on twistor spaces. 

\begin{prop}
\label{propDECOM}
\mbox{}
\vspace{.1cm}

\noindent 1) Let $J,K \in \mathcal{T}(V,g)$. The following are orthogonal and $J,K-$invariant decompositions of $V$:
  \[V= Im[J,K] \oplus Ker(J+K) \oplus Ker(J-K)\]
 
  and 
\[V= V_{e_{1}} \oplus ... \oplus V_{e_{l}} \oplus V_{1} \oplus V_{-1},\]

where
\[ V_{\epsilon}= Ker (\{J,K\}-2\epsilon\mathbb{1}) \text{ and }  e_{i} \in (-1,1).\]

\vspace{.4cm}
\noindent
2) a) $Im[J,K]= V_{e_{1}} \oplus ... \oplus V_{e_{l}}.$\\ 

b) $Ker(J+K)=V_{1}$ and $Ker(J-K)=V_{-1}.$\\

c) $Ker[J,K]= Ker(J+K) \oplus Ker(J-K)$.

\vspace{.4cm}
\noindent
3) a) Each $V_{e_{i}}$, where $e_{i} \in (-1,1)$, is an $\mathbb{H}$-module.\\

     b) $Rank[J,K]=4k$.
\end{prop}
\begin{proof} Please see \cite{Gindi2}.
%
%
%
%
\end{proof}

To prove the above proposition, we used the following result:  
\begin{prop}
\label{propK'}
Let $J,K \in \mathcal{T}(V,g)$ such that $\{J,K\}=2e\mathbb{1}$, for $e\in(-1,1)$. Also let $K'=\frac{JK-e}{f}$, where $e^{2} +f^{2}=1$. Then $K'\in \mathcal{T}$, $\{J,K'\}=0$ and $\{K,K'\}=0$.
\end{prop}

Given the splittings in Proposition \ref{propDECOM}, we can use the fact that each $V_{e_{i}}$ is an $\mathbb{H}$-module to obtain a finer decomposition of $V$.

\begin{prop}
\label{propDECOM2}
 Let $J,K \in \mathcal{T}(V,g)$. $V$ admits the following orthogonal and $J,K-$invariant decomposition:
\[V= W^{1}_{e_{1}} \oplus ... \oplus W^{q}_{e_{q}} \oplus V_{1} \oplus V_{-1},\]
where 
\begin{itemize}
\item $W^{i}_{e_{i}}$ is either $\{0\}$ or a four dimensional $J,K-$invariant subspace such that $\{J,K\}|_{W^{i}_{e_{i}}}=2e_{i}\mathbb{1}$, for $e_{i}\in (-1,1)$
\item $V_{1}= Ker(J+K)$ and $V_{-1}= Ker(J-K)$.
\end{itemize}
\end{prop}

Either by using the above proposition or Proposition \ref{propDECOM}, it is straightforward to prove
\begin{prop}
\label{propISOOV}
$J$ and $K$ in $\mathcal{T}(V,g)$ induce the same orientations on $V$ if and only if $\frac{dimKer(J+K)}{2}$ is even.
\end{prop}

As another corollary of Proposition \ref{propDECOM}, we can derive the orthogonal representations of $iD_{\infty}$, which is the algebra generated by two complex structures $J_{0}$ and $K_{0}$ over $\mathbb{R}$. (These representations are by definition the ones where $J_{0}$ and $K_{0}$ act by orthogonal transformations with respect to some metric.)
\begin{prop}
The irreducible, orthogonal representations of $iD_{\infty}$ are:
\[ \mathbb{R}[t]/(p) \oplus K_{0}\mathbb{R}[t]/(p), \]
where $J_{0}K_{0}$ acts by $t$ and 
\begin{align*}
1) \ &p= t\pm1 \\
2) \ &p= (t-c)(t-\overline{c}), \text{ for } c \in \mathbb{C} -\{\pm 1\}, c\overline{c}=1.
\end{align*}
\end{prop}
\begin{rmk}
In \cite{Gindi3}, we not only derive the orthogonal representations of $iD_{\infty}$ but the indecomposable ones as well. 
\end{rmk}

\subsection{On Constructing Algebroids on Twistor Spaces}
\label{secMOT}
Let us now fix a $J \in \mathcal{T}(V,g).$ In this paper we will build algebroids with integrable distributions that pointwise  interact, as elements of $\mathfrak{o}_{\{K\}}$, with the decomposition
 \[V=  V_{e_{1}} \oplus ... \oplus V_{e_{l}} \oplus V_{1} \oplus V_{-1}\] of Proposition \ref{propDECOM}. For instance, a particular integrable distribution which is associated to an algebroid and which we will introduce in Section \ref{secID} is 
 \[\mathcal{D}^{\{J,\phi\}(J+\phi)}= \{B \in \mathfrak{o}_{\{\phi\}} \ | B: KerF \longrightarrow 0, ImF\longrightarrow ImF\},\] where $F=\{J,\phi\}(J+\phi).$ 
   To build other algebroids and distributions with rich interactions, we will first use certain  combinations of the maps $J+\phi$, $J-\phi$ and $[J,\phi]$. The resulting distributions will  interact with the $V_{1}$ and $V_{-1}$ at each $K \in \mathcal{T}$. To produce distributions that interact with the $V_{e_{i}}$ we will also be using in Sections \ref{secMA} and \ref{secID} polynomials in $\{J,\phi\}$ of the form 
   \begin{equation}
   \label{eqQ}
   Q= a_{k}\{J,\phi\}^{k}+ ...+a_{1}\{J,\phi\}+ a_{0}\mathbb{1},
   \end{equation}
where the $a_{i}$ are \textit{functions} on $\mathcal{T}$ that satisfy certain differential equations. 

In the following section we will present three types of skew algebroids on $\mathcal{T}$ that will be based on  the maps $J+\phi$ and $J-\phi$. In Section \ref{secMA} we will use the $Q$s to construct various other algebroids whose associated foliations will refine the ones given in the next section.   

\section{First Three Skew Algebroids and Foliations}
\label{secFTSA}
\subsection{The $\mathcal{T}^{\#}(J)$ and Schubert Cells}
\label{secSC}
Given a $J\in \mathcal{T}(V,g)$, we will first introduce natural subsets of $\mathcal{T}$ that we will show to be the leaves associated to three types of skew algebroids on the twistor space. 

\begin{defi}
\label{defSC}
Letting $J\in\mathcal{T}$, we define
\begin{align*} 
1) \ &  \mathcal{T}^{(m_{1},*)}(J)= \{K \in \mathcal{T} | \ dimKer(J + K) = 2m_{1} \} \\
2) \ &  \mathcal{T}^{(*,m_{-1})}(J)= \{K \in \mathcal{T} | \ dimKer(J - K) = 2m_{-1} \} \\
3) \ & \mathcal{T}^{(m_{1},m_{-1})}(J)= \mathcal{T}^{(m_{1},*)} \cap  \mathcal{T}^{(*,m_{-1})}.
\end{align*}
\end{defi}

\begin{nota} When referring to the above subsets, we will at times drop the ``$(J)$" factors and will also let 
$\mathcal{T}^{\#}(J)$ or $\mathcal{T}^{\#}$ stand for any one of them.\end{nota}

In \cite{Gindi2} we proved that the $\mathcal{T}^{(m_{1},*)}(J)$ and $\mathcal{T}^{(*,m_{-1})}(J)$ are isomorphic to Schubert cells in a homogeneous maximal isotropic space. Using this point of view, they can be shown to be orbits of a group action and are thus associated to certain algebroids. The objective  of this section is to demonstrate how to use the $\mathcal{T}$ point of view to associate algebroids to all of the $\mathcal{T^{\#}}$. The main reason for focusing on this viewpoint is that the discussion in Section \ref{secMOT} has very naturally led us to consider maps of the form  $a_{k}\{J,\phi\}^{k}+ ...+a_{1}\{J,\phi\}+ a_{0}\mathbb{1}$ (Equation \ref{eqQ}) that we will use to establish many other algebroids. The foliations associated to these algebroids will richly refine the $\mathcal{T^{\#}}$ and in consequence the Schubert cells. The relations between maximal isotropics, $\mathcal{T}$ and the algebroids constructed here will be explored in \cite{Gindi4}.

We will now build two skew algebroids whose associated leaves are the $\mathcal{T}^{(m_{1},*)}(J)$ and the $\mathcal{T}^{(*,m_{-1})}(J)$. We will consider the $\mathcal{T}^{(m_{1},m_{-1})}(J)$ in Section \ref{secTSA}. 

\begin{nota}
Throughout this paper, given $(V,g)$, we will denote the adjoint of $A\in EndV$ by $A^{*}$.
\end{nota}
\subsection{First two Skew Algebroids} 
Let $J\in \mathcal{T}(V)$ and consider the map 
\begin{align*}
 \delta : \mathfrak{gl}(&\mathbb{V})  \longrightarrow T\mathcal{T} \\
 A & \longrightarrow [SA+A^{*}S, \phi],
 \end{align*}
 where $S \in \{J+\phi, J-\phi \} $.  \\ 

 \begin{defi}
 We define the following brackets on the sections of $\mathfrak{gl}(\mathbb{V})$.

 \vspace{.1cm}
 \noindent 1)\  When $S=J+\phi$: 
\[ [A,B]^{\delta}= dB_{\delta(A)} + \frac{1}{2}J\{A(J-\phi), J\}B -\frac{1}{2}B\delta(A)\phi -(A \leftrightarrow B).\]\\
 2) \  When  $S=J-\phi$: 
\[[A,B]^{\delta}= dB_{\delta(A)} + \frac{1}{2}J\{A(J+\phi), J\}B -\frac{1}{2}B\delta(A)\phi -(A \leftrightarrow B). \]

 \end{defi}
 
 We have then proved: 
 \begin{thm}
Given $S \in \{J+\phi, J-\phi \} $, $(\mathfrak{gl}(\mathbb{V}),[\cdot,\cdot]^{\delta}, \delta)$ is a skew algebroid on $\mathcal{T}(V)$.
\end{thm}
\begin{proof}
The proof will be given in \cite{Gindi4}; please see Remark \ref{rmkUPDATED}.
\end{proof}

\begin{rmk}
In \cite{Gindi4}, we will determine whether the above yields Lie algebroids on $\mathcal{T}(V)$; please see Section \ref{secFW1}.
\end{rmk}

Consequently, the $Im\delta$ is an integrable distribution on $\mathcal{T}$. It can be described as follows.

\begin{prop}  Given $S \in \{J+\phi, J-\phi \}$, 
\[Im \delta=  \{ B\in \mathfrak{o}_{\{\phi\}} \ | B: KerS \longrightarrow ImS\}.\]
\end{prop}

\begin{thm}
%
Consider the skew algebroids $(\mathfrak{gl}(\mathbb{V}),[\cdot,\cdot]^{\delta}, \delta)$ for $S=J+\phi$ and $S= J-\phi$. Also let $K \in \mathcal{T}(V)$ and $2m_{\pm 1}=dimKer(J\pm K)$.

\vspace{.2cm}
\noindent 1) When $S= J+\phi$,  $\mathcal{T}^{(m_{1},*)}(J)$ is the associated leaf that passes through $K$ and when $S= J-\phi$, it is $\mathcal{T}^{(*,m_{-1})}(J)$.   

\vspace{.2cm}
\noindent 2) These leaves are complex submanifolds of $\mathcal{T}(V)$. 
\end{thm}

\begin{defi}  
    \mbox{}
\vspace{.1cm}
    \begin{itemize}
    \item[1)] We will define $T_{\phi}\mathcal{T}^{(m_{1},*)}(J)$ to be the distribution that when evaluated at $K\in \mathcal{T}(V)$ is $T_{K}\mathcal{T}^{(m_{1}|_{K},*)}(J)$, where $2m_{1}|_{K}=dimKer(J+K)$. In this notation $m_{1}$ is a (non-smooth) function on $\mathcal{T}$ while in Definition \ref{defSC} it is a number.\\
    \item[2)]  Similarly, $T_{\phi}\mathcal{T}^{(*,m_{-1})}(J)$ is the distribution that when evaluated at $K\in \mathcal{T}(V)$ is $T_{K}\mathcal{T}^{(*, m_{-1}|_{K})}(J)$, where $2m_{-1}|_{K}=dimKer(J-K)$.
    \end{itemize}
 
    \end{defi}

\subsection{The Third Skew Algebroid and $\mathcal{T}^{(m_{1},m_{-1})}(J)$}
\label{secTSA}
Let us consider the $\mathcal{T}^{(m_{1},m_{-1})}= \mathcal{T}^{(m_{1},*)} \cap  \mathcal{T}^{(*,m_{-1})}$.

\begin{prop}
\begin{align*}
& 1)\  \mathcal{T}^{(m_{1},m_{-1})} \text{ is nonempty if and only if } n-m_{1}-m_{-1}=2k\ \  (k \in \mathbb{Z}_{\geq 0}).\\
& 2) \ \text{If } n-m_{1}-m_{-1}=2k, \text{ for } k \in \mathbb{Z}_{\geq 0}, \text{ then } \mathcal{T}^{(m_{1},*)} \text{ and } \mathcal{T}^{(*,m_{-1})}   \\ 
& \ \ \text{ intersect transversally}. \\
& 3) \  \mathcal{T}^{(m_{1},m_{-1})} \text{is a complex submanifold of } \mathcal{T}.
\end{align*}
\end{prop}
\begin{proof} Please see \cite{Gindi2}.
\end{proof}
Given these properties, we will now construct a skew algebroid whose associated leaves are the $\mathcal{T}^{(m_{1},m_{-1})}(J)$. 

To begin, consider the map
 \begin{align*}
 \sigma : \mathfrak{gl}(&\mathbb{V})  \longrightarrow T\mathcal{T} \\
 A & \longrightarrow [SAT-TA^{*}S, \phi],
 \end{align*}
 where $S=J+\phi$ and $T=J-\phi$.

\begin{defi} We define the following bracket on the sections of $\mathfrak{gl}(\mathbb{V})$:
\[ [A,B]^{\sigma}= dB_{\sigma(A)}+ \frac{1}{2}AS\phi S\{B,J\} -\frac{1}{2}\{A,J\}T\phi TB -(A \leftrightarrow B). \]
\end{defi}

 We have then proved: 
 \begin{thm}
$(\mathfrak{gl}(\mathbb{V}),[\cdot,\cdot]^{\sigma}, \sigma)$ is a skew algebroid on $\mathcal{T}(V)$.
\end{thm}

The $Im\sigma$ is then an integrable distribution on $\mathcal{T}$. It can be described as follows.

\begin{prop}
\[ Im\sigma= \{ B\in \mathfrak{o}_{\{\phi\}} \ | B: Ker(J+\phi) \longrightarrow Im(J+\phi), \ Ker(J-\phi) \longrightarrow Im(J-\phi) \}. \]
\end{prop}

\begin{thm}
Consider the skew algebroid $(\mathfrak{gl}(\mathbb{V}),[\cdot,\cdot]^{\sigma}, \sigma)$ and let $K \in \mathcal{T}(V)$ and $2m_{\pm1}=dimKer(J\pm K)$. $\mathcal{T}^{(m_{1},m_{-1})}(J)$ is the associated leaf that passes through $K$. Moreover, it is a complex submanifold of $\mathcal{T}.$
\end{thm}

\begin{defi}  
    \mbox{}
\vspace{.1cm}
   We will define $T_{\phi}\mathcal{T}^{(m_{1},m_{-1})}(J)$ to be the distribution that when evaluated  at $K\in \mathcal{T}(V)$ is $T_{K}\mathcal{T}^{(m_{1}|_{K},m_{-1}|_{K})}(J)$, where $2m_{\pm1}|_{K}=dimKer(J \pm K)$.       In this notation $m_{1}$ and $m_{-1}$ are (non-smooth) functions on $\mathcal{T}$ while in Definition \ref{defSC} they are numbers.
    \end{defi}

\begin{nota}
\label{notTT}
So far we have defined $T_{\phi}\mathcal{T}^{(m_{1},*)}(J)$, $T_{\phi}\mathcal{T}^{(*,m_{-1})}(J)$ and \ $T_{\phi}\mathcal{T}^{(m_{1},m_{-1})}(J)$. We will let $T_{\phi}\mathcal{T}^{\#}(J)$ or $T_{\phi}\mathcal{T}^{\#}$ stand for any of these distributions. 
\end{nota}

\section{More Algebroids: Refining the $\mathcal{T^{\#}}(J)$}
\label{secMA}
We will now introduce an abundance of other skew algebroids on $\mathcal{T}$. The foliations associated to these algebroids will subfoliate the $\mathcal{T^{\#}}(J)$ and, by the correspondence that was mentioned in Section \ref{secSC}, will subfoliate certain Schubert cells in a homogeneous maximal isotropic space.  The leaves of many of these foliations will be complex submanifolds of $\mathcal{T}$.
\begin{note}
From here on, though not always necessary, we will assume that $dimV\geq 4.$
\end{note}

\subsection{Some Background}
Before introducing the algebroids, we will need some background and notation. First consider

\begin{lemma}
\mbox{}

\noindent Let 
 \begin{itemize}
\item $dimV \geq 4$ 
\item $S\in \{ \mathbb{1}, J+\phi, J-\phi, [J,\phi] \}$
\item $\gamma \in \{\pm J, \pm \phi \}$.
\end{itemize}
1) There is a unique element $\tilde{\gamma} \in \{\pm J, \pm \phi \}$ that satisfies \[\gamma S= S\tilde{\gamma}.\]
2) Moreover, $\tilde{\gamma}$ is the unique element in $\{\pm J, \pm \phi \}$ that satisfies \[ S\gamma= \tilde{\gamma}S.\]
\end{lemma}
\begin{proof}
Use the relations: $J(J\pm \phi)= \pm(J\pm \phi)\phi$ and $\phi(J\pm \phi)= \pm(J\pm \phi)J$.
\end{proof}

Now consider
\begin{defi}
\mbox{}
\item[]
\vspace{2pt}
Let  
 \begin{itemize}
\item [a)]$S\in \{ \mathbb{1}, J+\phi, J-\phi, [J,\phi] \}$
\item [b)]$E= \{A\in \mathfrak{o}(\mathbb{V},g) | \ \gamma A= A \delta\}$,
\end{itemize}
where 
\[ \gamma \in \{\pm J, \pm \phi \} \text{\ \  and  \ \ } \delta \in \{+\gamma, -\gamma \}.\]
Then we define \[E^{S}= \{B\in \mathfrak{o}(\mathbb{V},g)| \ \tilde{\gamma} B= B \tilde{\delta\}},\]
where $\tilde{\gamma}$ and $\tilde{\delta}$ are the unique elements in $\{\pm J, \pm \phi\}$ that satisfy
  \[\gamma S=S \tilde{\gamma} \text{ \ \ and \ \ } S\delta= \tilde{\delta} S.\]
Moreover, when $E=\mathfrak{o}(\mathbb{V},g)$, we define $E^{S}=\mathfrak{o}(\mathbb{V},g).$ 
\end{defi}
\begin{example}
If $S=J+\phi$ and $E=\mathfrak{o}_{\{\phi\}}(\mathbb{V})$ then $E^{S}=\mathfrak{o}_{\{J\}}(\mathbb{V})$. 
\end{example}
The main property that the sections of $E^{S}$ satisfy is the following.

\begin{prop}
\label{propES}
\mbox{}
\item[]
\vspace{2pt}
Let  
 \begin{itemize}
\item [a)]$S\in \{ \mathbb{1}, J+\phi, J-\phi, [J,\phi] \}$
\item [b)]$E  \in \{\mathfrak{o}(\mathbb{V},g), \mathfrak{u}(\mathbb{V},J), \mathfrak{u}(\mathbb{V},\phi), \mathfrak{o}_{\{J\}}(\mathbb{V}), \mathfrak{o}_{\{\phi\}}(\mathbb{V})\}$.
\end{itemize}
If $B\in  \Gamma(E^{S})$ then $SBS\in \Gamma(E)$. 
\end{prop}

\subsection{The Algebroids}
\label{secALRES}
We will now introduce our major classes of skew algebroids on $\mathcal{T}(V,g)$.

\begin{thm}
\label{thmALM}
\mbox{}
\newline
\indent
Let
\begin{align*}
 \psi : E&^{S}  \longrightarrow T\mathcal{T} \\
 A & \longrightarrow [FAF, \phi],
 \end{align*}
 
 where 
\begin{itemize}
\item[a)] $F=QS \in \Gamma(\mathfrak{gl}(\mathbb{V}))$ is such that 
 \begin{itemize}
 \item[$\bullet$] $Q^{*}=Q$
 \item[$\bullet$] $[Q,J]=0$ and $[Q, \phi]= 0$
 \item[$\bullet$] $S \in \{ \mathbb{1}, J+\phi, J-\phi, [J,\phi] \}$
 \end{itemize}
\vspace{5pt}
 \item[b)] $E \in \{\mathfrak{o}(\mathbb{V},g),  \mathfrak{o}_{\{\phi\}}(\mathbb{V}), \mathfrak{u}(\mathbb{V},J), \mathfrak{u}(\mathbb{V},\phi)\}.$
 \end{itemize}
 
Moreover, suppose that 
\[ dQ_{[FCF,\phi]}= Q(dQ_{[SCS,\phi]})Q,\]
$\forall C \in \Gamma(E^{S}).$

Then there exists a skew bracket $[\cdot,\cdot]_{F}$ on $\Gamma(E^{S})$ such that $(E^{S},[\cdot,\cdot]_{F},\psi)$ is a skew algebroid on $\mathcal{T}(V)$.
\end{thm}
\begin{rmk}
Please see \cite{Gindi4} for the definition of the bracket $[\cdot,\cdot]_{F}$.
\end{rmk}

We thus see that the above algebroids depend on a choice of a $Q$ that satisfies certain differential equations. To construct examples of such $Q$, we will use the following proposition:
\begin{prop}
\mbox{}
\newline
\indent
Let 
\begin{itemize}
\item $Q= a_{k}\{J,\phi\}^{k}+...+a_{1}\{J,\phi\}+ a_{0}\mathbb{1}$
\item $C \in \Gamma(\mathfrak{o}(\mathbb{V},g))$ 
\end{itemize}
such that each $a_{i} \in C^{\infty}(\mathcal{T})$  satisfies 
\begin{equation*}
d{a_i}_{[QCQ,\phi]}=0 \ \ \ \ \  \text{ and }  \ \ \ \ \  d{a_i}_{[C,\phi]}=0.
\end{equation*}

Then \[ dQ_{[QCQ,\phi]}= Q(dQ_{[C,\phi]})Q.\]
\end{prop} 

Using this proposition and the above theorem, we have constructed the following skew algebroids on $\mathcal{T}(V)$.
\begin{thm}
\label{thmALP}
\mbox{}
\newline
\indent
Let
\begin{align*}
 \psi : E&^{S}  \longrightarrow T\mathcal{T} \\
 A & \longrightarrow [FAF, \phi],
 \end{align*}
 
 where 
\begin{itemize}
\item[a)] $F=QS \in \Gamma(\mathfrak{gl}(\mathbb{V}))$ is such that
 \begin{itemize}
 \item[$\bullet$] $Q= a_{k}\{J,\phi\}^{k}+...+a_{1}\{J,\phi\}+ a_{0}\mathbb{1}$, \  $a_{i} \in C^{\infty}(\mathcal{T})$
 \item[$\bullet$] $S \in \{ \mathbb{1}, J+\phi, J-\phi, [J,\phi] \}$
 \end{itemize}
\vspace{5pt}
 \item[b)] $E \in \{\mathfrak{o}(\mathbb{V},g),  \mathfrak{o}_{\{\phi\}}(\mathbb{V}), \mathfrak{u}(\mathbb{V},J), \mathfrak{u}(\mathbb{V},\phi)\}.$
 \end{itemize}
 
Moreover, suppose that 
\[ d{a_i}_{[SCS,\phi]}=0,\]
$\forall C \in \Gamma(E^{S}).$

Then there exists a skew bracket $[\cdot,\cdot]_{F}$ on $\Gamma(E^{S})$ such that $(E^{S},[\cdot,\cdot]_{F},\psi)$ is a skew algebroid on $\mathcal{T}(V)$.
\end{thm}

\subsection{Examples of the $\{a_{i}\}$: Symmetric Polynomials in the Eigenvalues of $\{J,\phi\}$}
\label{secSPE}
 Given the setup of Theorem \ref{thmALP}, note that if the $\{a_{i}\}$ are constant functions on $\mathcal
 {T}$ then $(E^{S},[\cdot,\cdot]_{F},\psi)$ is automatically a skew algebroid. The purpose of this section is to produce nonconstant functions on $\mathcal{T}$ that satisfy the differential conditions of Theorem \ref{thmALP} when $E=\mathfrak{u}(\mathbb{V},J)$; those that satisfy the conditions of Theorem \ref{thmALP} when $E=  \mathfrak{o}_{\{\phi\}}(\mathbb{V})$ will be presented in \cite{Gindi4}.

To build these functions, first consider 
\begin{prop}
\label{propSCS}
\mbox{}
\newline
Let 
\begin{itemize}
\item[a)] $F=QS \in \Gamma(\mathfrak{gl}(\mathbb{V}))$ such that
 \begin{itemize}
 \item[$\bullet$] $Q= a_{k}\{J,\phi\}^{k}+...+a_{1}\{J,\phi\}+ a_{0}\mathbb{1}$, \  $a_{i} \in C^{\infty}(\mathcal{T})$
 \item[$\bullet$] $S \in \{ \mathbb{1}, J+\phi, J-\phi, [J,\phi] \}$
 \end{itemize}
\vspace{5pt}
 \item[b)] $E=\mathfrak{u}(\mathbb{V},J).$
 \end{itemize}
 
If  
\[ d{a_i}_{[B,\phi]}=0, \ \forall B \in \Gamma(E)\]
 then 
\[ d{a_i}_{[SCS,\phi]}=0,\ \forall C \in \Gamma(E^{S}).\]

\end{prop}
\begin{proof}
This follows from the result given in Proposition \ref{propES} that if $C\in \Gamma(E^{S})$ then $SCS\in \Gamma(E)$.
\end{proof}

The goal then is to construct functions $a_{i}$ that satisfy $d{a_i}_{[B,\phi]}=0, \forall B \in \mathfrak{u}(V,J).$ If we consider the $U(J)$ action on $\mathcal{T}(V)$ given by $A \cdot  L = ALA^{-1}$ then this is equivalent to producing functions that are constant on each $U(J)$ orbit.
\begin{prop}
\label{propFJC}
Let $f_{j} \in C^{\infty}(\mathcal{T})$ be defined via
\[det(\lambda \mathbb{1}- \{J,\phi\})= \lambda^{2n}-f_{1}\lambda^{2n-1}+...-f_{2n-1}\lambda+f_{2n},\]
so that $f_{j}|_{K}$ is the $j^{th}$ symmetric polynomial in the eigenvalues of $\{J,K\}$.\\

\vspace{-.2cm}
1) $f_{j}$ is constant on each $U(J)$ orbit in $\mathcal{T}.$\\

2) $d{f_j}_{[B,\phi]}=0, \ \forall B \in \mathfrak{u}(V,J).$
\end{prop} 
\begin{proof}
This follows from the fact that if $A\in U(J)$ and $K \in \mathcal{T}$ then $\{J,AKA^{-1}\}=A\{J,K\}A^{-1}$.
\end{proof}
Combining the above results with Theorem \ref{thmALP} leads to
\begin{thm}
\label{thmALSP}
\mbox{}
\newline
\noindent
Let
\begin{align*}
 \psi: E&^{S}  \longrightarrow T\mathcal{T} \\
 A & \longrightarrow [FAF, \phi],
 \end{align*}
 
 where 
\begin{itemize}
\item[a)] $F=QS \in \Gamma(\mathfrak{gl}(\mathbb{V}))$ is such that 
 \begin{itemize}
 \item[$\bullet$] $Q= a_{k}\{J,\phi\}^{k}+...+a_{1}\{J,\phi\}+ a_{0}\mathbb{1}$, \  $a_{i} \in C^{\infty}(\mathcal{T})$
 \item[$\bullet$] $S \in \{ \mathbb{1}, J+\phi, J-\phi, [J,\phi] \}$
 \end{itemize}
\vspace{5pt}
 \item[b)] $E =\mathfrak{u}(\mathbb{V},J).$
 \end{itemize}
 Moreover, suppose
 \[a_{i} \in \mathbb{R}[f_{1},...,f_{2n}], \]
 where each $f_{j}\in C^{\infty}(\mathcal{T})$ is defined via
\[det(\lambda \mathbb{1}- \{J,\phi\})= \lambda^{2n}-f_{1}\lambda^{2n-1}+...-f_{2n-1}\lambda+f_{2n}.\]

\vspace{.2cm}
\noindent 
Then there exists a skew bracket $[\cdot,\cdot]_{F}$ on $\Gamma(E^{S})$ such that $(E^{S},[\cdot,\cdot]_{F},\psi)$ is a skew algebroid on $\mathcal{T}(V)$.

\end{thm}

\begin{rmk} In \cite{Gindi4}, we will use the skew algebroids of Sections \ref{secFTSA} and \ref{secMA}  to build Lie algebroids and Poisson structures on $\mathcal{T}$. Moreover, we will present additional real algebroids as well as holomorphic algebroids. Please see Section \ref{secFW}.
\end{rmk}

\section{Integrable Distributions}
\label{secID}
In this section, we will analyze the distributions associated to each of the above skew algebroids and will explore their integrability conditions. 
\subsection{The $Im\psi$}
\label{secID2}
To describe the distributions $Im\psi$, where $\psi$ is given in Theorem \ref{thmALM}, it is important to first know the following.
\begin{prop}
\mbox{}
\newline
\indent
Let $QS \in \Gamma(\mathfrak{gl}(\mathbb{V}))$ such that 
 \begin{itemize}
 \item[$\bullet$] $Q^{*}=Q$
 \item[$\bullet$] $[Q,J]=0$ and $[Q, \phi]= 0$
 \item[$\bullet$] $S \in \{ \mathbb{1}, J+\phi, J-\phi, [J,\phi] \}$.
\end{itemize}

Then for each $K\in \mathcal{T}(V)$,
\[ V= Ker QS|_{K} \oplus ImQS|_{K}\]
is an orthogonal and $J, K-$invariant splitting.
\end{prop}

Now consider 
\begin{prop}
\mbox{}
\newline
\indent
Let 
\begin{itemize}
\item[a)] $F=QS \in \Gamma(\mathfrak{gl}(\mathbb{V}))$ such that 
 \begin{itemize}
 \item[$\bullet$] $Q^{*}=Q$
 \item[$\bullet$] $[Q,J]=0$ and $[Q, \phi]= 0$
 \item[$\bullet$] $S \in \{ \mathbb{1}, J+\phi, J-\phi, [J,\phi] \}$
 \end{itemize}
\vspace{5pt}
 \item[b)] $E \in \{\mathfrak{o}(\mathbb{V},g), \mathfrak{o}_{\{\phi\}}(\mathbb{V}),  \mathfrak{u}(\mathbb{V},J), \mathfrak{u}(\mathbb{V},\phi)\}.$
\end{itemize}
Then for each $K\in \mathcal{T}(V)$, \[\{FAF \ | A\in E^{S}|_{K}\} = \{B \in E|_{K} \ | B: KerF \longrightarrow 0, \ ImF \longrightarrow ImF \},\] where $F$ is to be evaluated at $K$.  
\end{prop} 

Using this, we have proved
\begin{thm}
\label{thmDS}
\mbox{}
\newline
\indent
Let
\begin{align*}
 \psi : E&^{S} \longrightarrow T\mathcal{T} \\
 A & \longrightarrow [FAF, \phi],
 \end{align*}
 
 where 
\begin{itemize}
\item[a)] $F=QS \in \Gamma(\mathfrak{gl}(\mathbb{V}))$ is such that
 \begin{itemize}
 \item[$\bullet$] $Q^{*}=Q$
 \item[$\bullet$] $[Q,J]=0$ and $[Q, \phi]= 0$
 \item[$\bullet$] $S \in \{ \mathbb{1}, J+\phi, J-\phi, [J,\phi] \}$
 \end{itemize}
\vspace{5pt}
 \item[b)] $E \in \{\mathfrak{o}(\mathbb{V},g), \mathfrak{o}_{\{\phi\}}(\mathbb{V}), \mathfrak{u}(\mathbb{V},J), \mathfrak{u}(\mathbb{V},\phi)\}.$
 \end{itemize}
 \vspace{.5cm}

1) For $E \in \{\mathfrak{o}(\mathbb{V},g), \mathfrak{o}_{\{\phi\}}(\mathbb{V})\}$,
\[Im\psi= \mathcal{D}^{QS}:=\{B \in \mathfrak{o}_{\{\phi\}} \ | B: KerQS \longrightarrow 0, ImQS\longrightarrow ImQS\}. \]

2) For $E= \mathfrak{u}(\mathbb{V},J)$,
 \[Im\psi= \mathcal{D}^{QS}_{\mathfrak{u}(J)}:=\{[B,\phi] \in \mathfrak{o}_{\{\phi\}} \ | B \in \mathfrak{u}(V,J), \ B : KerQS \longrightarrow 0, ImQS\longrightarrow ImQS\}. \]

3) For $E= \mathfrak{u}(\mathbb{V},\phi)$,
\[Im\psi= 0.\]
\end{thm}

\begin{rmk}
In the above expressions for $Im\psi$, both sides of the equations should be evaluated at a point $K \in \mathcal{T}(V)$.
\end{rmk}

When working with the above distributions, it is useful to have some other notation:

\begin{defi}
Let $J\in \mathcal{T}(V,g)$ and let $V= W_{1} \oplus W_{2}$ be an orthogonal and $J$-invariant splitting. 
Then define

\[\mathfrak{o}^{W_{2}}_{\{J\}} = \{ A\in \mathfrak{o}_{\{J\}}(V) \ | A: W_{1}\longrightarrow 0, W_{2}\longrightarrow W_{2}\}\]
and 
 \[\mathfrak{u}(J)^{W_{2}} = \{ B\in \mathfrak{u}(V,J) \ | B: W_{1}\longrightarrow 0, W_{2}\longrightarrow W_{2}\}.\]
\end{defi}

Applying this to the distributions in Theorem \ref{thmDS}, we have:

\begin{prop}
The following holds true:
\begin{align*}
  1)\ & \mathcal{D}^{QS}= \mathfrak{o}_{\{\phi\}}^{ImQS} \\
  2)\  & \mathcal{D}^{QS}_{\mathfrak{u}(J)}= [\mathfrak{u}(J)^{ImQS}, \phi],
  \end{align*}
  where both sides of the equations are to be evaluated at a $K \in \mathcal{T}$.
\end{prop}

Given the above distributions, an important property of the $\mathcal{D}^{QS}$ is that they are  complex with respect to the complex structure $I_{\mathcal{T}}$ on $\mathcal{T}$, which was defined in Section \ref{secBTS}.  
\begin{prop}
The  distribution $\mathcal{D}^{QS}$ is closed under $I_{\mathcal{T}}$.
\end{prop}

\subsection{Integrability Conditions} 
\label{secICD}
We will now use the skew algebroids in Theorem \ref{thmALM} to determine conditions for the distributions in Theorem \ref{thmDS} to be integrable.

\begin{thm}
\label{thmICM}
\mbox{}
\newline
\indent
Let $QS \in \Gamma(\mathfrak{gl}(\mathbb{V}))$ such that 
 \begin{itemize}
 \item[$\bullet$] $Q^{*}=Q$
 \item[$\bullet$] $[Q,J]=0$ and $[Q, \phi]= 0$
 \item[$\bullet$] $S \in \{ \mathbb{1}, J+\phi, J-\phi, [J,\phi] \}$.\\
\end{itemize}

\indent
 1) $\mathcal{D}^{QS}$ is integrable if \[dQ_{QBQ}=Q(dQ_{B})Q, \ \ \ \forall B \in  \mathcal{D}^{S}.\] \  \\
 \indent
  2) $\mathcal{D}^{QS}_{\mathfrak{u}(J)}$ is integrable if \[dQ_{QCQ}=Q(dQ_{C})Q, \ \ \  \  \forall C \in  \mathcal{D}^{S}_{\mathfrak{u}(J)}.\] 

\end{thm}
 
 Using Theorem \ref{thmALP}, we have:
 \begin{thm}
 \label{thmICP}
 \mbox{}
\newline
\indent
Let $QS \in \Gamma(\mathfrak{gl}(\mathbb{V}))$ such that 
 \begin{itemize}
 \item[$\bullet$] $Q= a_{k}\{J,\phi\}^{k}+...+a_{1}\{J,\phi\}+ a_{0}\mathbb{1}$, \  $a_{i} \in C^{\infty}(\mathcal{T})$
 \item[$\bullet$] $S \in \{ \mathbb{1}, J+\phi, J-\phi, [J,\phi] \}$. \\
 \end{itemize}

\indent
1) $\mathcal{D}^{QS}$ is integrable if \[d{a_{i}}_{B}=0, \ \ \ \forall B \in  \mathcal{D}^{S}.\] \  \\
 \indent
  2) $\mathcal{D}^{QS}_{\mathfrak{u}(J)}$ is integrable if \[d{a_{i}}_{C}=0, \ \ \  \  \forall C \in  \mathcal{D}^{S}_{\mathfrak{u}(J)}.\] 
 \end{thm}
 
Noting that  $\mathcal{D}^{S}$ and $\mathcal{D}^{S}_{\mathfrak{u}(J)}$ are always integrable, we can reformulate the above result as follows:

  \begin{thm}
 \mbox{}
\newline
\indent
Let $QS \in \Gamma(\mathfrak{gl}(\mathbb{V}))$  such that 
 \begin{itemize}
 \item[$\bullet$] $Q= a_{k}\{J,\phi\}^{k}+...+a_{1}\{J,\phi\}+ a_{0}\mathbb{1}$, \  $a_{i} \in C^{\infty}(\mathcal{T})$
 \item[$\bullet$] $S \in \{ \mathbb{1}, J+\phi, J-\phi, [J,\phi] \}$.
 \end{itemize}
\indent
\begin{align*}
 1) \ &\mathcal{D}^{QS} \text{ is integrable if each } a_{i} \text{ is constant along the leaves of the foliation } \\ & \text{determined by } \mathcal{D}^{S}.
\end{align*}
 \begin{align*}
  2) \ &\mathcal{D}^{QS}_{\mathfrak{u}(J)} \text{ is integrable if each } a_{i} \text{ is constant along the leaves of the foliation } \\ & \text{determined by } \mathcal{D}^{S}_{\mathfrak{u}(J)}. 
  \end{align*}
  
 \end{thm}

As in Section \ref{secSPE}, let us now consider functions $f_{j} \in C^{\infty}(\mathcal{T})$ defined via 
\[det(\lambda \mathbb{1}- \{J,\phi\})= \lambda^{2n}-f_{1}\lambda^{2n-1}+...-f_{2n-1}\lambda+f_{2n},\]
so that $f_{j}|_{K}$ is the $j^{th}$ symmetric polynomial in the eigenvalues of $\{J,K\}$. Note that it follows from Proposition \ref{propFJC} that $d{f_{j}}_{C}=0$ for all $C \in \mathcal{D}^{\mathbb{1}}_{\mathfrak{u}(J)}$ and hence, by Proposition \ref{propSCS}, for all $C \in \mathcal{D}^{S}_{\mathfrak{u}(J)}$.  By either then applying Theorem \ref{thmICP} or by using Theorem \ref{thmALSP} directly, we have

\begin{thm}
\label{thmICSPE}
\mbox{}
\newline
\indent
Let $QS \in \Gamma(\mathfrak{gl}(\mathbb{V}))$  such that
\begin{itemize}
 \item[$\bullet$] $Q= a_{k}\{J,\phi\}^{k}+...+a_{1}\{J,\phi\}+ a_{0}\mathbb{1}$, \  $a_{i} \in C^{\infty}(\mathcal{T})$
 \item[$\bullet$] $S \in \{ \mathbb{1}, J+\phi, J-\phi, [J,\phi] \}$.
 \end{itemize}
 Moreover, suppose
 \[a_{i} \in \mathbb{R}[f_{1},...,f_{2n}], \]
 where each $f_{j}\in C^{\infty}(\mathcal{T})$ is defined via
\[det(\lambda \mathbb{1}- \{J,\phi\})= \lambda^{2n}-f_{1}\lambda^{2n-1}+...-f_{2n-1}\lambda+f_{2n}.\]

\vspace{.2cm}
\noindent
Then $\mathcal{D}^{QS}_{\mathfrak{u}(J)}$ is an integrable distribution on $\mathcal{T}(V)$.
\end{thm}
\begin{rmk}
We will consider examples of the above setup in Section \ref{secSPEFOL}.
\end{rmk}


 \subsection{Relating the Distributions and Refinements}
 
 We will now describe the different ways that the above integrable distributions refine both each other and the distributions $T_{\phi}\mathcal{T}^{\#}(J)$ (see Notation \ref{notTT}). This will consequently lead to various foliations and subfoliations of $\mathcal{T}$ and $\mathcal{T}^{\#}(J)$ which we will explore in Sections \ref{secFCF} and \ref{secSCF}. 
 
The first set of refinements of the $T_{\phi}\mathcal{T}^{\#}(J)$ will be given in the next two propositions.

\begin{prop}
\label{propRDR1}
\mbox{}
\vspace{.2cm}
\begin{itemize}
\item[1)] $\mathcal{D}^{J+\phi} \subset T_{\phi}\mathcal{T}^{(m_{1},*)}(J)$.\\
\item[2)] $\mathcal{D}^{J-\phi} \subset T_{\phi}\mathcal{T}^{(*,m_{-1})}(J)$.\\
\item[3)] $\mathcal{D}^{[J,\phi]} \subset T_{\phi}\mathcal{T}^{(m_{1},m_{-1})}(J)$.
\end{itemize}
\end{prop}

\begin{prop}
\label{propRDR2}
\[\mathcal{D}^{\mathbb{1}}_{\mathfrak{u}(J)} \subset T_{\phi}\mathcal{T}^{\#}(J). \]
\end{prop}

The following gives more refinements:
\begin{prop}
\mbox{}
\vspace{.2cm}
\begin{itemize}
\item[1)] $\mathcal{D}^{S}_{\mathfrak{u}(J)} \subset \mathcal{D}^{\mathbb{1}}_{\mathfrak{u}(J)}$.
\item[2)]  $\mathcal{D}^{S}_{\mathfrak{u}(J)} \subset \mathcal{D}^{S}$.
\item[3)] $ \mathcal{D}^{S}_{\mathfrak{u}(J)}=  \mathcal{D}^{\mathbb{1}}_{\mathfrak{u}(J)} \cap\mathcal{D}^{S} $.
\end{itemize}
\end{prop}

To obtain even more refinements, we will now turn on a $Q$. First consider the following proposition where no integrability conditions are imposed on $\mathcal{D}^{QS}$ and $\mathcal{D}^{QS}_{\mathfrak{u}(J)}$.
\begin{prop}
\mbox{}
\newline
\indent
Let $QS \in \Gamma(\mathfrak{gl}(\mathbb{V}))$ such that 
 \begin{itemize}
 \item[$\bullet$] $Q^{*}=Q$
 \item[$\bullet$] $[Q,J]=0$ and $[Q, \phi]= 0$
 \item[$\bullet$] $S \in \{ \mathbb{1}, J+\phi, J-\phi, [J,\phi] \}$.
 \end{itemize}
 The following holds true: 

\hspace{3.5cm}
\begin{tikzpicture}
  \matrix (m) [matrix of math nodes,row sep=1em,column sep=1em,minimum
width=2em] {
    \mathcal{D}^{QS} & \mathcal{D}^{S} \\
    \mathcal{D}^{QS}_{\mathfrak{u}(J)} & \mathcal{D}^{S}_{\mathfrak{u}(J)}. \\
  };
\path[-stealth, auto] (m-1-1) edge[draw=none] node [sloped,
auto=false,allow upside down] {$\subset$} (m-1-2);

\path[-stealth, auto] (m-2-1) edge[draw=none] node [sloped,
auto=false,allow upside down] {$\subset$} (m-1-1);

\path[-stealth, auto] (m-2-2) edge[draw=none] node [sloped,
auto=false,allow upside down] {$\subset$} (m-1-2);

\path[-stealth, auto] (m-2-1) edge[draw=none] node [sloped,
auto=false,allow upside down] {$\subset$} (m-2-2);
\end{tikzpicture}  
\end{prop}

Imposing integrability conditions on $\mathcal{D}^{QS}$ and $\mathcal{D}^{QS}_{\mathfrak{u}(J)}$  leads to 

\begin{prop}
\mbox{}
\newline
\indent
Let $QS \in \Gamma(\mathfrak{gl}(\mathbb{V}))$  such that 
 \begin{itemize}
 \item[$\bullet$] $Q^{*}=Q$
 \item[$\bullet$] $[Q,J]=0$ and $[Q, \phi]= 0$
 \item[$\bullet$] $S \in \{ \mathbb{1}, J+\phi, J-\phi, [J,\phi] \}$.
\end{itemize}

1) If $dQ_{QBQ}= Q(dQ_{B})Q$ for all $B \in \mathcal{D}^{S}$ then the following holds true for the specified integrable distributions:

\begin{center}
\begin{tikzpicture}
  \matrix (m) [matrix of math nodes,row sep=1em,column sep=1em,minimum
width=2em] {
    \mathcal{D}^{QS} & \mathcal{D}^{S} \\
    \mathcal{D}^{QS}_{\mathfrak{u}(J)} & \mathcal{D}^{S}_{\mathfrak{u}(J)}. \\
  };
\path[-stealth, auto] (m-1-1) edge[draw=none] node [sloped,
auto=false,allow upside down] {$\subset$} (m-1-2);

\path[-stealth, auto] (m-2-1) edge[draw=none] node [sloped,
auto=false,allow upside down] {$\subset$} (m-1-1);

\path[-stealth, auto] (m-2-2) edge[draw=none] node [sloped,
auto=false,allow upside down] {$\subset$} (m-1-2);

\path[-stealth, auto] (m-2-1) edge[draw=none] node [sloped,
auto=false,allow upside down] {$\subset$} (m-2-2);
\end{tikzpicture}
\end{center}

2) If $dQ_{QCQ}= Q(dQ_{C})Q$ for all $C \in \mathcal{D}^{S}_{\mathfrak{u}(J)}$ then the following is a series of integrable distributions that refine each other:
\[ \mathcal{D}^{QS}_{\mathfrak{u}(J)} \subset \mathcal{D}^{S}_{\mathfrak{u}(J)} \subset   \mathcal{D}^{S}.\] 
\end{prop}

We have a similar result for the case when $Q= a_{k}\{J,\phi\}^{k}+...+a_{1}\{J,\phi\}+ a_{0}\mathbb{1}$:
\begin{prop}
 \mbox{}
\newline
\indent
Let $QS \in \Gamma(\mathfrak{gl}(\mathbb{V}))$ such that 
 \begin{itemize}
 \item[$\bullet$] $Q= a_{k}\{J,\phi\}^{k}+...+a_{1}\{J,\phi\}+ a_{0}\mathbb{1}$, \  $a_{i} \in C^{\infty}(\mathcal{T})$
 \item[$\bullet$] $S \in \{ \mathbb{1}, J+\phi, J-\phi, [J,\phi] \}$.
 \end{itemize}

1) If $d{a_{i}}_{B}=0$ for all $B \in \mathcal{D}^{S}$ then the following holds true for the specified integrable distributions:

\begin{center}
\begin{tikzpicture}
  \matrix (m) [matrix of math nodes,row sep=1em,column sep=1em,minimum
width=2em] {
    \mathcal{D}^{QS} & \mathcal{D}^{S} \\
    \mathcal{D}^{QS}_{\mathfrak{u}(J)} & \mathcal{D}^{S}_{\mathfrak{u}(J)}. \\
  };
\path[-stealth, auto] (m-1-1) edge[draw=none] node [sloped,
auto=false,allow upside down] {$\subset$} (m-1-2);

\path[-stealth, auto] (m-2-1) edge[draw=none] node [sloped,
auto=false,allow upside down] {$\subset$} (m-1-1);

\path[-stealth, auto] (m-2-2) edge[draw=none] node [sloped,
auto=false,allow upside down] {$\subset$} (m-1-2);

\path[-stealth, auto] (m-2-1) edge[draw=none] node [sloped,
auto=false,allow upside down] {$\subset$} (m-2-2);
\end{tikzpicture}
\end{center}

2) If $d{a_{i}}_{C}=0$ for all $C \in \mathcal{D}^{S}_{\mathfrak{u}(J)}$ then the following is a series of integrable distributions that refine each other:
\[ \mathcal{D}^{QS}_{\mathfrak{u}(J)} \subset \mathcal{D}^{S}_{\mathfrak{u}(J)} \subset   \mathcal{D}^{S}.\] 
\end{prop}

\begin{rmk}
Other relations will be given in \cite{Gindi4}.

\end{rmk}


\section{First Class of Foliations: The $\mathcal{T}^{QS}_{K}$ }
\label{secFCF}
 In this section, we will describe the foliations associated to the integrable distributions $\mathcal{D}^{QS}$, as given in Theorem \ref{thmICP}. We will also show how they refine each other and the $\mathcal{T}^{\#}(J)$ of Section \ref{secFTSA}. In Section \ref{secSCF}, we will analyze the foliations associated to the $\mathcal{D}^{QS}_{\mathfrak{u}(J)}$.
 
  \subsection{The $\bf{\mathcal{T}^{QS}_{K}}$}
 Let $QS \in \Gamma(\mathfrak{gl}(\mathbb{V}))$ such that 
 \begin{itemize}
 \item[$\bullet$] $Q= a_{k}\{J,\phi\}^{k}+...+a_{1}\{J,\phi\}+ a_{0}\mathbb{1}$, \  $a_{i} \in C^{\infty}(\mathcal{T})$
 \item[$\bullet$] $S \in \{ \mathbb{1}, J+\phi, J-\phi, [J,\phi] \}$
 \end{itemize}
and suppose 
 \begin{equation}
 \label{eqFCF}
 d{a_{i}}_{B}=0, \ \ \ \forall B \in  \mathcal{D}^{S}.
 \end{equation}
 By Theorem \ref{thmICP},  
 \begin{align*}
 \mathcal{D}^{QS}&=\mathfrak{o}_{\{\phi\}}^{ImQS} \\&=\{C \in \mathfrak{o}_{\{\phi\}} \ | C: KerQS \longrightarrow 0, ImQS\longrightarrow ImQS\} 
 \end{align*}
 is an integrable distribution on $\mathcal{T}$.

 To describe the corresponding leaf at $K \in \mathcal{T}(V)$, consider the following orthogonal and  $K$-invariant splitting
\[ V= KerQS|_{K} \oplus ImQS|_{K} \]
and the associated one for $K$:  \[K= K^{a} \oplus K^{b}.\]
 
 \begin{defi}Given the above splittings, define \[\mathcal{T}^{QS}_{K}= \{K'= K^{a}\oplus K'^{b} \ |  K'^{b} \in \mathcal{T}(ImQS|_{K}), ImQS|_{K'}=ImQS|_{K}\}.\]
 \end{defi}
 
\begin{thm}
\label{thmFCF}
The connected component of $\mathcal{T}^{QS}_{K}$ that contains $K\in \mathcal{T}(V)$ is the leaf that is associated to the distribution $\mathcal{D}^{QS}$ and that passes through $K$. Moreover, it is a complex submanifold of $\mathcal{T}(V)$.
\end{thm} 

\begin{nota}
The connected component of $\mathcal{T}^{QS}_{K}$ that contains $K\in \mathcal{T}(V)$ will be denoted by 
$\accentset{\circ}{\mathcal{T}}^{QS}_{K}$.
\end{nota}

\begin{rmk}  The $\mathcal{T}^{QS}_{K}$ are generally not connected; see Example \ref{exTQ}. \end{rmk}

\begin{rmk}
To prove Theorem \ref{thmFCF}, we first proved it for the case when $Q=\mathbb{1}$.  We then used, in particular, the fact that the $a_{i}$ are locally constant on the $\mathcal{T}^{S}_{K}$, which follows from Theorem \ref{thmFCF} and Equation \ref{eqFCF}. There are several other important steps in the proof; we will explain them in detail in \cite{Gindi4}. (See also Remark \ref{rmkOOS}.)
\end{rmk}
To determine the manifold type of ${\mathcal{T}^{QS}_{K}}$, consider the following embedding:

 \begin{align*}
 \eta^{QS}_{K} : \mathcal{T}(Im&QS|_{K})  \longrightarrow \mathcal{T}(V) \\
 &L^{b}  \longrightarrow K^{a} \oplus L^{b},
 \end{align*}
 so that $Im \eta^{QS}_{K}= \{K^{a} \oplus L^{b} \ | L^{b}\in \mathcal{T}(ImQS|_{K})\}$.

We then have:
\begin{prop}
\label{propTOS}
$\mathcal{T}^{QS}_{K}$ is an open subset of $Im \eta^{QS}_{K}$.
\end{prop}

\begin{rmk}
To prove this proposition, we need to understand the behavior of the $\{a_{i}\}$ on $Im \eta^{QS}_{K}$. The details will be given in \cite{Gindi4}.
\end{rmk}
\begin{rmk}
\label{rmkOOS}
Proposition \ref{propTOS} is one of the steps needed to prove Theorem \ref{thmFCF}.
\end{rmk}

Here is an example of the $\mathcal{T}^{QS}_{K}$ when $dimV=4.$
\begin{example}
\label{exTQ}
\mbox
\newline

Let
\begin{itemize}
 \item $dimV=4$
 \item $I,J\in \mathcal{T}(V,g)$ such that $\{I,J\}=0$ 
 \item $L=IJ$
 \item $Q=\{J,\phi\}(\{J,\phi\}- 2e_{0}\mathbb{1})$, where $e_{0} \in (0,1)$.
 \end{itemize}
 Consider now  
 \begin{align*}
 \mathcal{T}^{+}&= \{L' \in \mathcal{T}(V) \ | L' \text{ induces the same orientation as that of } J\}\\ &= \{aI+bJ +cL \ | a^{2}+b^{2}+c^{2}=1\}.
 \end{align*}
  We will presently describe the foliation of $\mathcal{T}^{+}$ that is associated to the integrable distribution $\mathcal{D}^{Q}$. \\
  
 1) Let $K=aI +cL \in \mathcal{T}^{+}$. As $V=KerQ|_{K}$, $\mathcal{T}^{Q}_{K}=\{K\}$.\\
 
 2) Let $K=aI -e_{0}J+ cL \in \mathcal{T}^{+}$. Since $\{J,K\}=2e_{0}\mathbb{1}$, $V=KerQ|_{K}$ and hence $\mathcal{T}^{Q}_{K}=\{K\}$.\\

3) Let $K= aI +bJ +cL \in \mathcal{T}^{+}$  such that $b\notin \{0, -e_{0}\}$. Since $\{J,K\}=-2b\mathbb{1}$, $V= ImQ|_{K}$. Thus 
\begin{align*}
\mathcal{T}^{Q}_{K}&=\{K' \in \mathcal{T}(V) \ | \{J,K'\}(\{J,K'\}-2e_{0}\mathbb{1}) \text{ is invertible}\} \\&=\{K' \in \mathcal{T}(V) \ | \{J,K'\} \text{ does not have eigenvalues } 0 \text{ and } 2e_{0} \}\\
&=\mathcal{T}(V)- \{ dI+eJ+fL \in \mathcal{T}^{+} \ | e=0 \text{ or } e=-e_{0}\}.
 \end{align*}
 As this is not connected, we need to take its connected component 
 $\accentset{\circ}{\mathcal{T}}^{Q}_{K}$. Consider the following cases: \\
 
 a) If $b\in (0,1]$ then 
 \[\accentset{\circ}{\mathcal{T}}^{Q}_{K}=\{dI +eJ+ fL \in \mathcal{T}^{+} \ | \ e \in(0,1]  \}. \]
 
  b) If $b\in (-e_{0},0)$ then 
  \[\accentset{\circ}{\mathcal{T}}^{Q}_{K}=\{dI +eJ+ fL \in \mathcal{T}^{+} \ | \ e \in(-e_{0},0)  \}.\]
 
  c) If $b\in [-1,-e_{0})$ then 
   \[\accentset{\circ}{\mathcal{T}}^{Q}_{K}=\{dI +eJ+ fL \in \mathcal{T}^{+} \ | \ e \in[-1,-e_{0})  \}. \] \qed
 \end{example}

\subsection{Relating the Foliations and Refinements}
Before giving more examples, we will now relate the foliations of $\mathcal{T}$ that we have introduced so far.
We begin by providing another description of the $\mathcal{T}^{S}_{K}$.

\begin{prop}
\label{propSAD}
\begin{align*}
1)& \ \mathcal{T}^{J+ \phi}_{K}=\{ L \in \mathcal{T}(V) \ | Ker(J+L)=Ker(J+K)\}.\\ 
2)& \ \mathcal{T}^{J- \phi}_{K}=\{ L \in \mathcal{T}(V) \ | Ker(J-L)=Ker(J-K)\}.\\
3)& \  \mathcal{T}^{[J,\phi]}_{K}= \mathcal{T}^{J+\phi}_{K} \cap \mathcal{T}^{J- \phi}_{K}.
\end{align*}
\end{prop}

As a corollary, we will show that the $\mathcal{T}^{S}_{K}$ refine the appropriate $\mathcal{T^{\#}}(J)$, as  expected from Proposition \ref{propRDR1}.

\begin{prop}
\begin{align*}
1)\ & \mathcal{T}^{J+\phi}_{K} \subset \mathcal{T}^{(m_{1},*)}(J), \text{ where } 2m_{1}=dimKer (J+K). \\
2)\ & \mathcal{T}^{J-\phi}_{K} \subset \mathcal{T}^{(*,m_{-1})}(J), \text{ where } 2m_{-1}=dimKer (J-K). \\
3)\ & \mathcal{T}^{[J,\phi]}_{K} \subset \mathcal{T}^{(m_{1},m_{-1})}(J), \text{ where } 2m_{\pm 1}=dimKer (J\pm K). 
\end{align*}
\end{prop}

\begin{rmk}
In \cite{Gindi4}, we will explore how, for instance, the $\mathcal{T}^{J+\phi}_{K}$ fit together inside the $\mathcal{T}^{(m_{1},*)}(J)$.
\end{rmk}

To obtain further refinements, we will now use the $Qs$: 

\begin{prop}
$\mathcal{T}^{QS}_{K} \subset \mathcal{T}^{S}_{K}$. 
\end{prop}

\begin{rmk}
To prove this proposition, we would need to introduce decompositions of $V$ that are compatible with the ones used to define $\mathcal{T}^{S}_{K} $ and $\mathcal{T}^{QS}_{K}$. Please see \cite{Gindi4}. 
\end{rmk}

\subsection{Controlling the Manifold Type of $\mathcal{T}^{QS}_{K}$}
\label{secCMTQS}
We can control the manifold type of $\mathcal{T}^{QS}_{K}$, in particular control its dimension, by choosing $Q|_{K}$ to have roots that lie in a specific subset of the eigenvalues of $\{J,K\}$. We demonstrate this in the following example which can be appropriately generalized. 

\begin{example}
\label{exCMT}
Using Proposition \ref{propDECOM}, let $J,K \in \mathcal{T}(V,g)$  such that  
\[V=V_{e_{1}} \oplus V_{e_{2}}\oplus V_{e_{3}} \oplus V_{1} \oplus V_{-1},\]
where $V_{\epsilon}= Ker (\{J,K\}-2\epsilon\mathbb{1})$ and $\ e_{i} \in (-1,1)$.

We will then choose
\begin{itemize}
\item  $Q= (\{J,\phi \}-2e_{1}\mathbb{1})(\{J,\phi \}-2e_{2}\mathbb{1})$
\item $S=J+ \phi$.
\end{itemize}
To describe $\mathcal{T}^{QS}_{K}$, first split 
\begin{align*}&V= KerQS|_{K} \oplus ImQS|_{K}\\
                     & K= K^{a} \oplus K^{b}\\
                     & J= J^{a} \oplus J^{b},
\end{align*}
noting that
\[ KerQS|_{K}= V_{e_{1}} \oplus V_{e_{2}} \oplus V_{1}\] 
and 
\[ImQS|_{K}= V_{e_{3}} \oplus V_{-1}.\]

Now consider $K'=K^{a}\oplus K'^{b}$, where $K'^{b} \in \mathcal{T}(ImQS|_{K})$. Then $K'$ is, by definition, an element of $\mathcal{T}^{QS}_{K}$ if and only if 
 \[ImQS|_{K'}=ImQS|_{K}.\]
 Since \[QS|_{K'}= 0 \oplus (\{J^{b},K'^{b}\}-2e_{1}\mathbb{1})(\{J^{b},K'^{b} \}-2e_{2}\mathbb{1})(J^{b}+K'^{b}), \]
 this is equivalent to the condition that 
 \[ (\{J^{b},K'^{b}\}-2e_{1}\mathbb{1})(\{J^{b},K'^{b} \}-2e_{2}\mathbb{1})(J^{b}+K'^{b})\\ \]
is invertible as an element of $End(ImQS|_{K})$.

Hence,
\[
\mathcal{T}^{QS}_{K}=\{K^{a}\oplus K'^{b}\ | K'^{b} \in \mathcal{T}(ImQS|_{K}), \{J^{b},K'^{b}\} \text{ does not have eigenvalues } 2e_{1}, 2e_{2} \text{ and } 2\}. 
\]
\qed
\end{example}

\section{Second Class of Foliations: The $\mathcal{T}^{QS}_{U(J),K}$}
\label{secSCF}
In this section, we will describe the leaves of the foliation associated to the distribution $\mathcal{D}^{QS}_{\mathfrak{u}(J)}$ of Theorem \ref{thmICP}.  We will then show how they lead to foliations of the leaves  introduced above.

\vspace{.1cm}
We will begin with the case when $QS=\mathbb{1}$.
\subsection{Classifying the $U(J)$ Orbits}
 Since $\mathcal{D}^{\mathbb{1}}_{\mathfrak{u}(J)}= [\mathfrak{u}(J), \phi],$ the associated leaves are the orbits of the $U(J)$ action on $\mathcal{T}(V,g)$ that is given by conjugation. The point here is to use the background on a pair of complex structures of Section \ref{secBOPCS} to classify all of these orbits. 
Before doing so, we will characterize them based on the eigenvalues of $\{J,\phi\}$. 

%

\begin{thm}
$K$ and $L$ in $\mathcal{T}(V,g)$ belong to the same $U(J)$ orbit if and only if $\{J,K\}$ and  $\{J,L\}$ have the same eigenvalues with multiplicity.
\end{thm} 
\begin{proof} The proof uses the decomposition in Proposition \ref{propDECOM2} and will be given in \cite{Gindi4}.
\end{proof}

We will now classify all of the $U(J)$ orbits by using the decompositions given in Proposition \ref{propDECOM}.

\begin{thm}
\label{thmCUJO}
Let $J,K \in \mathcal{T}(V,g)$ such that 
\[V=V_{e_{1}}\oplus...\oplus V_{e_{l}}\oplus V_{1}\oplus V_{-1}, \]
where
\begin{itemize}
\item $V_{\epsilon}=Ker(\{J,K\}-2\epsilon\mathbb{1})$ and $e_{i}\in (-1,1)$
\item $dimV_{e_{i}}=4k_{i}$, $dimV_{1}=2m_{1}$ and $dimV_{-1}=2m_{-1}$.
\end{itemize}
Then \[U(J)\cdot K\cong U(n)/ Sp(k_{1}) \times ...\times Sp(k_{l}) \times U(m_{1}) \times U(m_{-1}).\]
\end{thm}
\begin{proof}
As $U(J)\cdot K \cong U(J)/ U(J)\cap U(K),$ we need only analyze $U(J)\cap U(K)$. First note that it preserves each subspace in the decomposition given in the theorem.
To describe $U(J)\cap U(K)|_{V_{e_{i}}}$, let $K'= \frac{JK-e_{i}}{f_{i}}|_{V_{e_{i}}}$, where $e_{i}^{2}+f_{i}^{2}=1$. Then by Proposition \ref{propK'} $K' \in \mathcal{T}(V_{e_{i}},g)$ and $\{J,K'\}=0$. Hence on $V_{e_{i}}$, $U(J)\cap U(K)=U(J)\cap U(K') \cong Sp(k_{i})$. The theorem then follows.
\end{proof}

\subsection{The $\bf{\mathcal{T}^{QS}_{U(J), K}}$}
Now let $QS \in \Gamma(\mathfrak{gl}(\mathbb{V}))$ such that 
 \begin{itemize}
 \item[$\bullet$] $Q= a_{k}\{J,\phi\}^{k}+...+a_{1}\{J,\phi\}+ a_{0}\mathbb{1}$, \  $a_{i} \in C^{\infty}(\mathcal{T})$
 \item[$\bullet$] $S \in \{ \mathbb{1}, J+\phi, J-\phi, [J,\phi] \}$ 
 \end{itemize}
 and suppose  
 \begin{equation}
 \label{eqSCF}
 d{a_{i}}_{B}=0, \ \ \  \  \forall B \in  \mathcal{D}^{S}_{\mathfrak{u}(J)}.
 \end{equation}
 By Theorem \ref{thmICP},
\begin{align*}
\mathcal{D}^{QS}_{\mathfrak{u}(J)}&=[\mathfrak{u}(J)^{ImQS}, \phi]\\&=\{[C,\phi] \in \mathfrak{o}_{\{\phi\}} \ | C \in \mathfrak{u}(J), \ C : KerQS \longrightarrow 0, ImQS\longrightarrow ImQS\} 
\end{align*}
is an integrable distribution on $\mathcal{T}(V,g)$.

To describe the corresponding leaf at $K \in \mathcal{T}$,  consider the following orthogonal and $J,K-$invariant splitting 
   \[ V= KerQS|_{K} \oplus ImQS|_{K}\]
 together with its associated splittings
\begin{align*}
& K=K^{a} \oplus K^{b}\\
&J=J^{a} \oplus J^{b}.
\end{align*}
The leaf passing through $K$ will be related to  $U(J^{b})\cdot K^{b}$, which is the orbit of $K^{b}$ associated to the action of $U(J^{b})$ on $\mathcal{T}(ImQS|_{K})$ given by conjugation.  
\begin{defi}
Given the above splittings, define 
\[\mathcal{T}^{QS}_{U(J), K}= \{ K'=K^{a}\oplus K'^{b} \ |  K'^{b} \in U(J^{b})\cdot K^{b}\}.\]
\end{defi}

\begin{thm}
\label{thmSCF}
$\mathcal{T}^{QS}_{U(J), K}$ is the leaf that is associated to the distribution $\mathcal{D}^{QS}_{\mathfrak{u}(J)}$ and that passes through $K\in \mathcal{T}(V)$. 
\end{thm}
\begin{rmk}
To prove Theorem \ref{thmSCF}, we first proved it for the case when $Q=\mathbb{1}$. We then used, in particular, the fact that the $a_{i}$ are constant on the $\mathcal{T}^{S}_{U(J), K}$, which follows from Theorem \ref{thmSCF} and Equation \ref{eqSCF}. There are several other important steps in the proof; we will explain them in detail in \cite{Gindi4}.
\end{rmk}

\subsection{Relations between the Foliations and Refinements}
Before giving examples of the $\mathcal{T}^{QS}_{U(J), K}$, we will demonstrate how they refine each other and the leaves of the previous sections.

First note that the $U(J)$ orbits refine the $\mathcal{T}^{\#}(J)$, as expected from Proposition \ref{propRDR2}.
\begin{prop} If $K \in \mathcal{T}^{\#}(J)$ then
$\mathcal{T}^{\mathbb{1}}_{U(J), K} \subset \mathcal{T}^{\#}(J).$
\end{prop}

 We can then use the $\mathcal{T}^{S}_{U(J), K}$ to foliate the $U(J)$ orbits:
\begin{prop}
$\mathcal{T}^{S}_{U(J), K} \subset \mathcal{T}^{\mathbb{1}}_{U(J), K}$.
\end{prop}
The $\mathcal{T}^{S}_{U(J), K}$ also foliate the $\mathcal{T}^{S}_{K}$. In fact, we have
\begin{prop}
\[\mathcal{T}^{S}_{U(J), K}= \mathcal{T}^{\mathbb{1}}_{U(J), K} \cap \mathcal{T}^{S}_{K}. \]
\end{prop}
Combining this with Proposition \ref{propSAD} yields
\begin{prop}
\mbox
\newline
\vspace{.1cm}

1)\ $\mathcal{T}^{J+\phi}_{U(J), K}=\{L \in U(J) \cdot K \ | Ker(J+L)= Ker(J+K)\}$.\\

2)\ $\mathcal{T}^{J-\phi}_{U(J), K}=\{L \in U(J) \cdot K \ | Ker(J-L)= Ker(J-K)\}$.\\

3)\ $\mathcal{T}^{[J,\phi]}_{U(J), K}=\mathcal{T}^{J+\phi}_{U(J), K} \cap \mathcal{T}^{J-\phi}_{U(J), K}$.
\end{prop}

To obtain further refinements, we will now turn on a $Q$:
\begin{prop}
$\mathcal{T}^{QS}_{U(J), K}\subset \mathcal{T}^{S}_{U(J), K}$.
\end{prop}
\begin{rmk}
To prove this proposition, we would need to introduce decompositions of $V$ that are compatible with the ones used to define $\mathcal{T}^{S}_{U(J), K} $ and $\mathcal{T}^{QS}_{U(J), K}$. Please see \cite{Gindi4}. 
\end{rmk}
\begin{rmk}
In Section \ref{secSRF}, we will summarize the relations and refinements between all the foliations that we have introduced in this paper. 
\end{rmk}
We will now give our main examples of the $\mathcal{T}^{QS}_{U(J), K}.$
\subsection{Symmetric Polynomials in the Eigenvalues of $\{J,\phi\}$} 
\label{secSPEFOL}
As in Sections \ref{secSPE} and \ref{secICD}, let us consider functions $f_{j} \in C^{\infty}(\mathcal{T})$ defined via 
\[det(\lambda \mathbb{1}- \{J,\phi\})= \lambda^{2n}-f_{1}\lambda^{2n-1}+...-f_{2n-1}\lambda+f_{2n},\]
so that $f_{j}|_{K}$ is the $j^{th}$ symmetric polynomial in the eigenvalues of $\{J,K\}$. Note by Proposition \ref{propFJC}, $d{f_{j}}_{B}=0$ for all $B \in  \mathcal{D}^{\mathbb{1}}_{\mathfrak{u}(J)}$ and hence for all $B \in \mathcal{D}^{S}_{\mathfrak{u}(J)}$. By then applying Theorems \ref{thmICSPE} and \ref{thmSCF}, we have 
\begin{thm}
\mbox{}
\newline
\indent
Let $QS \in \Gamma(\mathfrak{gl}(\mathbb{V}))$ such that
\begin{itemize}
 \item[$\bullet$] $Q= a_{k}\{J,\phi\}^{k}+...+a_{1}\{J,\phi\}+ a_{0}\mathbb{1}$, \  $a_{i} \in C^{\infty}(\mathcal{T})$
 \item[$\bullet$] $S \in \{ \mathbb{1}, J+\phi, J-\phi, [J,\phi] \}$.
 \end{itemize}
 Moreover, suppose
 \[a_{i} \in \mathbb{R}[f_{1},...,f_{2n}], \]
 where each $f_{j}\in C^{\infty}(\mathcal{T})$ is defined via
\[det(\lambda \mathbb{1}- \{J,\phi\})= \lambda^{2n}-f_{1}\lambda^{2n-1}+...-f_{2n-1}\lambda+f_{2n}.\]

\vspace{.2cm}
\noindent
Then $\mathcal{T}^{QS}_{U(J), K}$ is the leaf that is associated to the integrable distribution $\mathcal{D}^{QS}_{\mathfrak{u}(J)}$ and that passes through $K\in\mathcal{T}(V)$.
\end{thm}

Given the data in this theorem, let $Q=(\{J,\phi\}- b_{1}\mathbb{1})...(\{J,\phi\}- b_{k}\mathbb{1})$, where $b_{i} \in \mathbb{R}[f_{1},...,f_{2n}]$, and consider the leaf $\mathcal{T}^{QS}_{U(J), K}$. Since $f_{j}|_{K}$ is the $j^{th}$ symmetric polynomial in the eigenvalues of $\{J,K\}$, the manifold type of $\mathcal{T}^{QS}_{U(J), K}$ depends, in particular, on whether these eigenvalues satisfy certain systems of polynomial equations. Here is an example.

\begin{example}
\label{exSPE}
\mbox{}

\vspace{.2cm}
\noindent
Let 
\begin{itemize}
\item $J \in \mathcal{T}(V,g)$, where $dimV=12$
\item $Q= \{J,\phi\}- tr\frac{\{J,\phi\}}{4}$
\item $S=\mathbb{1}$.
\end{itemize}
\end{example}
\noindent
We will analyze the associated foliation of \[\mathcal{T}^{+}= \{L \in \mathcal{T}(V)\ | L \text{ induces the same orientation as that of } J\}.\]

First note that if $K \in \mathcal{T}^{+}$ then it follows from Propositions \ref{propDECOM} and \ref{propISOOV} that the dimension of each eigenspace of $\{J,K\}$ is a multiple of four. Hence we can split $V$ as follows:
\[V= W^{1}_{a_{1}} \oplus W^{2}_{a_{2}} \oplus W^{3}_{a_{3}},\]
where $W^{i}_{a_{i}}$ is a four dimensional $J,K-$invariant subspace such that $\{J,K\}|_{W^{i}_{a_{i}}}=2a_{i}\mathbb{1},$ for $a_{i} \in [-1,1].$ Thus 
\[Q|_{K} = \{J,K\}- (2a_{1} + 2a_{2}+ 2a_{3})\mathbb{1} \] and 
\begin{equation} 
\label{eqSPLIT}
Ker Q|_{K}= Ker(a_{2}+a_{3})\mathbb{1}_{W^{1}_{a_{1}}} \oplus Ker(a_{1}+a_{3})\mathbb{1}_{W^{2}_{a_{2}}}\oplus Ker(a_{1}+a_{2})\mathbb{1}_{W^{3}_{a_{3}}}.
\end{equation}
To describe $\mathcal{T}^{Q}_{U(J), K}$, consider 
\[V= KerQ|_{K} \oplus ImQ|_{K}\]
and  
\begin{align*}
    &K= K_{1} \oplus K_{2}\\
   &J= J_{1} \oplus J_{2}.
\end{align*}
Then, as we know,  
\begin{align}
\label{eqORBIT}
\mathcal{T}^{Q}_{U(J), K}&= \{ K_{1}\oplus K'_{2} \ | K'_{2} \in U(J_{2})\cdot K_{2}\} \notag\\
                                         &\cong  U(J_{2})\cdot K_{2}.
\end{align}
The point here is that using Theorem \ref{thmCUJO}, we can determine the type of manifold this is. It will depend on the values and the multiplicities of the $a_{i}$ as well as on whether these eigenvalues satisfy a certain polynomial equation. Consider the following cases:\\
                   
\noindent
1) If there exist $i,j \in \{1,2,3\}$ such that  $i\neq j$ and \[a_{i}+a_{j}=0\] then for $k\in \{1,2,3\}$, where $k\neq i$ and $k\neq j$,
\[2a_{k}=2a_{i}+2a_{j}+2a_{k}= tr\frac{\{J,K\}}{4}.\] Hence $Q|_{K}= \{J,K\}- 2a_{k}\mathbb{1}$ and $KerQ|_{K}= Ker(\{J,K\}- 2a_{k}\mathbb{1})$. 

\vspace{.2cm}
a) Suppose $V= W^{1}_{a} \oplus W^{2}_{b} \oplus W^{3}_{-b}$ where $a\neq \pm b$. Then $KerQ|_{K}=W^{1}_{a}$ and $ImQ|_{K}=W^{2}_{b} \oplus W^{3}_{-b}$. Using Theorem \ref{thmCUJO} and Equation \ref{eqORBIT}, we have:
\vspace{.2cm}

   \hspace{.2cm} i) If $b=1$ or $-1$ then    
     \[\mathcal{T}^{Q}_{U(J), K} \cong U(4)/U(2) \times U(2).\]
     
    \hspace{.2cm} ii) If $b=0$ then 
          \[\mathcal{T}^{Q}_{U(J), K} \cong U(4)/Sp(2).\]
          
     \hspace{.2cm} iii) If $b\notin \{0, \pm1\}$  then          
         \[\mathcal{T}^{Q}_{U(J), K} \cong U(4)/Sp(1) \times Sp(1).\] 
   
   \vspace{.05cm}
    b) Suppose $V= W^{1}_{a} \oplus W^{2}_{a} \oplus W^{3}_{-a}$ where $a=0$. Since $KerQ|_{K}=V$, $\mathcal{T}^{Q}_{U(J), K}=\{K\}$.\\ 
    
    c)  Suppose $V= W^{1}_{a} \oplus W^{2}_{a} \oplus W^{3}_{-a}$ where $a\neq 0$. Then $KerQ|_{K}=W^{1}_{a} \oplus W^{2}_{a}$ and $ImQ|_{K}= W^{3}_{-a}.$
    
      \vspace{.2cm}
     \hspace{.2cm} i) If $a=\pm1$ then $\mathcal{T}^{Q}_{U(J), K}=\{K\}$.
      
     \hspace{.2cm} ii) If $a\neq \pm1$ then 
           \begin{align*}
           \mathcal{T}^{Q}_{U(J), K}& \cong U(2)/Sp(1)\\
                                                    &\cong S^{1}.
           \end{align*}
 \noindent
  2) If there does not exist an $i,j \in \{1,2,3\}$ such that  $i\neq j$ and \[a_{i}+a_{j}=0\]   then by Equation \ref{eqSPLIT}, $Ker Q|_{K}=\{0\}$ and hence 
   $\mathcal{T}^{Q}_{U(J), K}=U(J)\cdot K.$ All of these orbits have been classified in Theorem \ref{thmCUJO}.
    \section{Summary of Relations between the Foliations }
    \label{secSRF}
   In this section we will summarize some of the relations between the foliations that were introduced in this paper.
   
   We begin with some refinements of the $\mathcal{T}^{\#}(J)$ using only the $``S"$ terms:
   \begin{prop}
\begin{align*}
1)\ & \mathcal{T}^{J+\phi}_{K} \subset \mathcal{T}^{(m_{1},*)}(J), \text{ where } 2m_{1}=dimKer (J+K). \\
2)\ & \mathcal{T}^{J-\phi}_{K} \subset \mathcal{T}^{(*,m_{-1})}(J), \text{ where } 2m_{-1}=dimKer (J-K). \\
3)\ & \mathcal{T}^{[J,\phi]}_{K} \subset \mathcal{T}^{(m_{1},m_{-1})}(J), \text{ where } 2m_{\pm 1}=dimKer (J\pm K). 
\end{align*}
\end{prop}
   
    \begin{prop} Let $K \in \mathcal{T}^{\#}(J).$ \\
    \vspace{.1cm}
\hspace*{.4cm}1) \ $\mathcal{T}^{S}_{U(J),K} \subset \mathcal{T}^{\mathbb{1}}_{U(J), K} \subset \mathcal{T}^{\#}(J).$ 

\vspace*{.05cm}
 \noindent \hspace*{.31cm} 2) \ $\mathcal{T}^{S}_{U(J), K}= \mathcal{T}^{\mathbb{1}}_{U(J), K} \cap \mathcal{T}^{S}_{K}.$
\end{prop}

    We then obtain further refinements by using the $Q$s:
    
\begin{prop}
 \mbox{}
\newline
\indent
Let $QS \in \Gamma(\mathfrak{gl}(\mathbb{V}))$ such that 
 \begin{itemize}
 \item[$\bullet$] $Q= a_{k}\{J,\phi\}^{k}+...+a_{1}\{J,\phi\}+ a_{0}\mathbb{1}$, \  $a_{i} \in C^{\infty}(\mathcal{T})$
 \item[$\bullet$] $S \in \{ \mathbb{1}, J+\phi, J-\phi, [J,\phi] \}$.
 \end{itemize}

1) If $d{a_{i}}_{B}=0$ $\forall B \in \mathcal{D}^{S}$ then the following yields foliations of $\mathcal{T}(V)$ that refine each other:

\begin{center}
\begin{tikzpicture}
  \matrix (m) [matrix of math nodes,row sep=1em,column sep=1em,minimum
width=2em] {
    \mathcal{T}^{QS}_{K} & \mathcal{T}^{S}_{K} \\
    \mathcal{T}^{QS}_{U(J),K} & \mathcal{T}^{S}_{U(J),K}. \\
  };
\path[-stealth, auto] (m-1-1) edge[draw=none] node [sloped,
auto=false,allow upside down] {$\subset$} (m-1-2);

\path[-stealth, auto] (m-2-1) edge[draw=none] node [sloped,
auto=false,allow upside down] {$\subset$} (m-1-1);

\path[-stealth, auto] (m-2-2) edge[draw=none] node [sloped,
auto=false,allow upside down] {$\subset$} (m-1-2);

\path[-stealth, auto] (m-2-1) edge[draw=none] node [sloped,
auto=false,allow upside down] {$\subset$} (m-2-2);
\end{tikzpicture}
\end{center}

2) If $d{a_{i}}_{C}=0$ $\forall C \in \mathcal{D}^{S}_{\mathfrak{u}(J)}$ then the following yields foliations of $\mathcal{T}(V)$ that refine each other:
\[ \mathcal{T}^{QS}_{U(J),K} \subset \mathcal{T}^{S}_{U(J),K} \subset   \mathcal{T}^{S}_{K}.\] 
\end{prop} Other relations were given above. In \cite{Gindi4} we will explore more relations and will study how the leaves fit together in $\mathcal{T}(V)$.

\section{Additional Results in Updated Version}
    \label{secFW}
 
 In this section, I describe some additional results that I will present in an updated version of this paper \cite{Gindi4}.      
 
   \subsection{Lie Algebroids, Poisson Structures and Holomorphicity}
    \label{secFW1}
     In \cite{Gindi4}, I will be using the skew algebroids of Sections \ref{secFTSA} and \ref{secMA} to build Lie algebroids and Poisson structures on $\mathcal{T}(V,g)$. I will also be constructing associated \textit{holomorphic} skew algebroids, Lie algebroids and Poisson structures.  
    
    Given these Lie algebroids, I will then study, for instance, their cohomology and will explore any associated Lie groupoids.
    
\subsection{Relation to Maximal Isotropics and $\mathbb{CP}^{3}$}
   As mentioned in Section \ref{secSC}, some of the $\mathcal{T}^{\#}(J)$ can be viewed as Schubert cells in a  homogeneous maximal isotropic space. In \cite{Gindi4}, I will study the other foliations and algebroids given in Sections \ref{secMA}, \ref{secFCF} and \ref{secSCF} from the maximal isotropic point of view. 
   
 Moreover, using the fact that $\mathcal{T}(V) \cong \mathbb{CP}^{3}$, where $dimV=6,$ I will transfer the algebroids in this paper to $\mathbb{CP}^{3}$ and will study the associated foliations.  
   
   \subsection{More Algebroids on Twistor Spaces}
   \mbox{}
   
   1) In \cite{Gindi4}, I will build other algebroids on twistor spaces that are certain ``direct sums" of the ones given in Section \ref{secMA}.
   
  2) The way that I constructed the algebroids in Sections \ref{secFTSA} and \ref{secMA} was to, in particular, fix one complex structure $J\in \mathcal{T}(V,g)$. In \cite{Gindi4}, I will build other algebroids by fixing \textit{two} complex structures $J,K \in \mathcal{T}(V,g)$. Both types of algebroids will have applications to bihermitian geometry \cite{Gindi5}.
  
  3) I have also produced many other algebroids on the twistor space $\mathcal{C}(V)=\{K \in EndV \ | K^{2}=-1 \}$. For example, I have built algebroids that are similar to the ones given in Theorem \ref{thmALM} but used an $F_{1}=Q_{1}S_{1}$  and an $F_{2}=Q_{2}S_{2}$ that are generally not equal. They will be presented in \cite{ Gindi4}.
  \subsection{Applications}
 The applications of the algebroids in this paper will be given in \cite{Gindi5} and \cite{Gindi6}.

\textsc{Department of Pure Mathematics, University of Waterloo, Waterloo, ON N2L 3G1, CANADA} \\

\textit{E-mail Address:} \texttt{steven.gindi@uwaterloo.ca} 

\end{large}
\end{document}